\documentclass[12pt]{amsart}

\usepackage{preamble}

\title[Topological Zeta Functions of Matroids]{Topological Zeta Functions of Matroids: \\ Operations and Computations}

\author{Dawit Mengesha}\address{Dawit Mengesha}\email{dym5466@psu.edu}
\author{Robert Miranda}\address{Robert Miranda}\email{robertmiranda@math.ucla.edu}
\author{Brian Sun}\address{Brian Sun}\email{bsun@berkeley.edu}

\begin{document}
\maketitle
\allowdisplaybreaks

\begin{abstract}
    The topological zeta function of a matroid is a rational function as well as a valuative invariant of the matroid, encoding rich combinatorial information. We analyze topological zeta functions of matroids from the vantage point of several matroid operations and operations on lattices of flats. We prove a clean recurrence relation for the M\"{o}bius inversion and use it to describe the topological zeta function of the truncation and free extension of a matroid in relation with that of the original matroid. We also characterize the Taylor coefficients of the topological zeta functions for matroids in terms of a matroid invariant, which we call the girth, and generalize an earlier result by Jensen--Kutler--Usatine.
\end{abstract}

\setcounter{tocdepth}{1}
\tableofcontents

\section{Introduction}
The topological zeta function of a complex polynomial is a singularity invariant of the corresponding hypersurface, introduced by Denef and Loeser \cite{denef1992caracteristiques}.  In \cite{feichtner2004incidence}, Feichtner and Kozlov developed the notion of combinatorial blowups of an atomic meet-semilattice, which enabled van der Veer \cite{van2019combinatorial} to define a topological zeta function for all such lattices, including geometric lattices. Let $M$ be a matroid on a ground set $E$. The topological zeta function $Z_M^\top(s)$ is a rational function associated to $M$, which is a specialization of the motivic zeta function of a matroid, introduced and studied in \cite{jensen2021motivic}. It encodes combinatorial data of the underlying matroid; for example, the first Taylor coefficient at $s=0$ is precisely $-|E|$ \cite[Theorem 1.12]{jensen2021motivic}. 

The topological zeta function $Z_M^{\top}(s)$ is defined as a function on the lattice of flats of $M$. As a result, we can take the M\"{o}bius inversion to obtain another rational function associated to $M$, which we denote by $Y_M(s)$. Our first main result states that $Y_M(s)$ satisfies a recurrence relation analogous to the one for $Z_M^{\top}(s)$, and we compute the M\"{o}bius inversion explicitly for uniform matroids.
\begin{thm}[Theorem~\ref{thm:upsilon-recurrence}]
    Let $M$ be a matroid over ground set $E$. Then
    \[
        Y_M(s) = \frac{-1}{|E|s + \rk(M)} \sum_{\substack{F \in \calL(M) \\ F \neq E}} (|E|s + \rk(F)) Y_{M|F}(s),
    \]
    where the sum is taken over all proper flats $F \in \calL(M)$.
\end{thm}

Our second result is to use the formula above for the $Y_M(s)$ to describe $Z_M^{\top}$ under the operations of truncation $\tr(M)$ and free extension $M + e$. 

\begin{thm}[Theorem~\ref{thm:truncation-zeta}]
    Let $M$ be a matroid over ground set $E$. The topological zeta function for the truncation $\tr(M)$ can be written in terms of $Z^{\top}_M(s)$ and its M\"{o}bius inversion $Y_M(s)$ as
    \[
        Z^{\top}_{\tr(M)}(s) = Z^{\top}_M(s) + \frac{1}{|E|s + \rk(M) - 1} Y_M(s).
    \]
\end{thm}

\begin{thm}[Theorem~\ref{thm:free-extension-zeta}]
    Let $M$ be a matroid over ground set $E$. The topological zeta function for the free extension $M + e$ can be written in terms of $Z^{\top}_M(s)$ and its M\"{o}bius inversion $Y_M(s)$ as
    \[
        Z^{\top}_{M+e}(s) = \frac{1}{s+1} \left(Z^{\top}_M(s) - \frac{s}{(|E|+1)s + \rk(M)} Y_M(s) \right).
    \]
\end{thm}

Our third main result is a characterization of the low order Taylor coefficients of a topological zeta function in the expansion at $s = 0$, in terms of the girth of the matroid.\footnote{In particular, this generalizes a result from an unpublished REU project led by Max Kutler on the Taylor coefficients of a uniform matroid.} This result also generalizes an earlier result in \cite[Theorem 1.12]{jensen2021motivic} to Taylor coefficients higher than $1$. 
\begin{thm}[Theorem~\ref{thm:girth}]
    Let $M$ be a matroid over a ground set $E$ of girth $g > 0$. Then for all integers $0 \leqslant k < g$, the derivative of the topological zeta function evaluated at $0$ is
    \[
        \left(\frac{d^k}{ds^k}Z^{\top}_M(s) \right) \bigg|_{s=0} = (-1)^k |E|^{\overline{k}},
    \]
    where $|E|^{\overline{k}} = |E|(|E|+1) \cdots(|E| + k-1)$ is the rising $k^{th}$ factorial of $|E|$.
\end{thm}

 One interpretation for the girth of the matroid is that it is the largest rank up to which truncation returns a uniform matroid. Thus our theorem is a special case of the following conjecture when the truncated matroid is uniform.

\begin{conjecture}[Conjecture~\ref{conj:truncation}]
    Let $M$ be a matroid of rank $r > 0$. Then for all integers $0 \leqslant k < r$ we have
    \[
        \left(\frac{d^k}{ds^k} Z^{\top}_M(s)\right) \bigg|_{s=0} = \left(\frac{d^k}{ds^k} Z^{\top}_{\tr(M)}(s) \right) \bigg|_{s=0}.
    \]
\end{conjecture}

A proof of this conjecture would give a combinatorial interpretation of the low order Taylor coefficients of the topological zeta function $Z_M^{\top}(s)$ in terms of the lattice of flats of $M$. In terms of the M\"{o}bius inversion of the topological zeta function, we can strengthen our conjecture.

\begin{conjecture}[Conjecture~\ref{conj:upsilon}]
    Let $M$ be a matroid of rank $r$. Then the Taylor expansion of $Y_M(s)$ centered at $s=0$ is
    \[
        Y_M(s) = (-1)^r |\calB(M)| s^r + O(s^{r+1}).
    \]
\end{conjecture}

\subsection{Acknowledgement} This project was completed as part of the 2020 Summer Undergraduate Mathematics Research at Yale (SUMRY) program. We are grateful to our mentors, Max Kutler and Shiyue Li, for their guidance throughout the project. We also thank all the organizers and participants for creating a stimulating and inspiring mathematical environment even as the world navigates the start of a pandemic.

\section{Preliminaries}
\label{sec:prelim}
We invite readers to see \cite{oxley2011book, white1987combinatorial} for general background on matroids. Here we outline our notation. Let $M$ be a loopless matroid on a ground set $E$. We set
\begin{align*}
    \calB(M) &:= \text{ the set of bases of } M, \\
    \calL(M) &:= \text{ the lattice of flats of } M.
\end{align*}
In particular, $\calL(M)_r$ will denote the sublattice consisting of all flats $F \in \calL(M)$ of rank $\rk_M(F) = r$, and $\calL(M)_{\leqslant r}$ all flats of rank $\rk_M(F) \leqslant r$. Given a flat $F \in \calL(M)$, we set
\begin{align*}
    M|F &\coloneqq \text{the restriction of } M \text{ at } F, \\
    M/F &\coloneqq \text{the contraction of } M \text{ at } F.
\end{align*}
Given a flag of flats
\[
    \calF = (\varnothing = F_0 \subsetneq F_1 \subsetneq \dotsc \subsetneq F_k = E),
\]
in $\calL(M)$, the \textbf{degeneration of $M$ at $\calF$} is the matroid
\[
    M_{\calF} = \bigoplus_{i=1}^k M | F_i / F_{i-1}.
\]
Finally, the characteristic polynomial of $M$ is denoted as $\chi_M(q)$.

Now we are ready to define the topological zeta function.
\begin{definition}[\cite{vanderveer2019, jensen2021motivic}]
\label{def:zeta-flags}
    Let $M$ be a matroid. The \textbf{topological zeta function} $Z_M^{\top}(s)$ of $M$ is defined as follows:
    \begin{itemize}
        \item If $M$ is the trivial matroid, then $Z_M^\top(s)\coloneqq 1$.
        \item If $M$ contains loops, then $Z_M^{\top}(s)\coloneqq 0$.
        \item If $M$ is loopless, then 
        \[
            Z_M^{\top}(s) := \sum_{\calF} \frac{\chi_{M_{\calF}}(q)}{(q-1)^{\len(\calF)}} \Bigg|_{q=1} \ \prod_{\substack{F \in \calF \\ F \neq \varnothing}} \frac{1}{|F|s + \rk(F)} \in \Q(s),
        \]
        where the sum is taken over all flags $\calF$ in the lattice of flats $\calL(M)$ and $\chi$ is the characteristic polynomial of the degeneration $M_{\calF}$.
    \end{itemize}
\end{definition}

To compute $Z^{\top}_M(s)$ for a given matroid, we will often use the following recurrence relation in terms of the topological zeta functions of the restrictions of $M$.

\begin{proposition}[\cite{jensen2021motivic}, Corollary 7.2]
\label{prop:zeta-recurrence}
    Let $M$ be a matroid over ground set $E$. Then
    \[
        Z^{\top}_M(s) = \frac{1}{|E|s + \rk(M)} \sum_{\substack{F \in \calL(M) \\ F \neq E}} \overline{\chi}_{M/F}(1) Z^{\top}_{M|F}(s),
    \]
    where the sum is taken over all proper flats $F \in \calL(M)$.
\end{proposition}

We also note that the topological zeta function is multiplicative over direct sums.

\begin{proposition}[\cite{vanderveer2019}, Theorem 5]
\label{prop:zeta-multiplicative}
    Let $M_1$ and $M_2$ be matroids. Then
    \[
        Z^{\top}_{M_1 \oplus M_2}(s) = Z^{\top}_{M_1}(s) Z^{\top}_{M_2}(s).
    \]
\end{proposition}

Using this decomposition, one easily obtains an explicit closed formula for the topological zeta function of any uniform matroid. \footnote{This result was also independently obtained in an unpublished REU project led by Max Kutler on the topological zeta functions for uniform matroids.}
\begin{proposition}
\label{prop:zeta-uniform}
The topological zeta function of the uniform matroid $U_{r,n}$ is
\[
    Z^{\top}_{U_{r,n}}(s) = \frac{1}{ns + r} \sum_{k=0}^{r-1} {n \choose k} {r - n \choose r-1 -k} \frac{1}{(s+1)^k}.
\]
In particular, the topological zeta function of $U_{n,n}$ is
\[
    Z^{\top}_{U_{n,n}}(s) = \frac{1}{(s+1)^n}.
\]
In addition, the Taylor series expansion of $Z^{\top}_{U_{r,n}}(s)$ around $0$ is
\[
    Z^{\top}_{U_{r,n}}(s) = \sum_{k \geqslant 0} a_k s^k,
\]
where
\[
    a_k = \begin{cases}
        (-1)^k \multiset{n}{k} & \quad \text{if $k \leqslant r$} \\
        -\frac{1}{r} \sum\limits_{i=1}^r \Big[ n \binom{r-1}{i-1} + r \binom{r-1}{i} \Big] a_{k-i} & \quad \text{if $k > r$}.
    \end{cases}
\]
    
\end{proposition}

\section{The M\"{o}bius inversion of the topological zeta function}
\label{sec:mobius}
This section is devoted to studying the M\"{o}bius inversion of the topological zeta function, which will prove to be a valuable tool in our computations. We will also see that it shares some properties of the topological zeta function.
\begin{definition}
\label{def:upsilon}
    Let $M$ be a matroid. The \text{M\"{o}bius inversion} $Y_M(s)$ of the topological zeta function $Z^{\top}_M(s)$ is
    \[
        Y_M(s) := \sum_{F \in \calL(M)} \mu(F,E) Z^{\top}_{M|F}(s),
    \]
    where $\mu$ is the M\"{o}bius function on the lattice of flats $\calL(M)$.
\end{definition}

Many of the result for topological zeta functions have an analogous statement for its M\"{o}bius inversion. For instance, we have a recurrence relation in terms of the restriction to proper flats, as in Proposition \ref{prop:zeta-recurrence}.

\begin{thm}
\label{thm:upsilon-recurrence}
    Let $M$ be a matroid over ground set $E$. Then
    \[
        Y_M(s) = \frac{-1}{|E|s + \rk(M)} \sum_{\substack{F \in \calL(M) \\ F \neq E}} (|E|s + \rk(F)) Y_{M|F}(s),
    \]
    where the sum is taken over all proper flats $F \in \calL(M)$.
\end{thm}

\begin{proof}
    From the M\"{o}bius inversion formula, Definition \ref{def:upsilon} becomes
    \[
        Y_M(s) = Z^{\top}_M(s) - \sum_{\substack{F \in \calL(M) \\ F \neq E}} Y_{M|F}(s).
    \]
    Now applying Proposition \ref{prop:zeta-recurrence} to evaluate $Z^{\top}_M(s)$, and writing each topological zeta function in terms of the M\"{o}bius inversion formula, we get
    \[
        Y_M(s) = \frac{1}{|E|s + \rk(M)} \left(\sum_{\substack{F \in \calL(M) \\ F \neq E}} \overline{\chi}_{M/F}(1) \sum_{\substack{G \in \calL(M) \\ G \subseteq F}} Y_{M|G}(s) \right) - \sum_{\substack{F \in \calL(M) \\ F \neq E}} Y_{M|F}(s).
    \]
    Working within the parentheses, exchanging the order in the double sum and applying \cite[Proposition 2.2]{jensen2021motivic} we have
    \begin{align*}
        \sum_{\substack{F \in \calL(M) \\ F \neq E}} \overline{\chi}_{M/F}(1) \sum_{\substack{G \in \calL(M)}} Y_{M|G}(s) &= \sum_{\substack{G \in \calL(M) \\ G \neq E}} Y_{M | G}(s) \sum_{\substack{F \in \calL(M) \\ G \subseteq F \subsetneq E}} \overline{\chi}_{M/F}(1) \\
        &= \sum_{\substack{G \in \calL(M) \\ G \neq E}} [\rk(M) - \rk(G)] \cdot Y_{M|G}(s).
    \end{align*}
    Replacing the indexing variable $G$ with $F$, and returning to the original expression, we conclude that
    \begin{align*}
        Y_M(s) &= \frac{1}{|E|s + \rk(M)} \sum_{\substack{F \in \calL(M) \\ F \neq E}} (\rk(M) - \rk(F)) Y_{M|F}(s) - \sum_{\substack{F \in \calL(M) \\ F \neq E}} Y_{M|F}(s) \\
        &= \frac{-1}{|E|s + \rk(M)} \sum_{\substack{F \in \calL(M) \\ F \neq E}} (|E|s + \rk(F)) Y_{M|F}(s).\qedhere
    \end{align*}
\end{proof}

Using this, we give an explicit closed formula for the M\"{o}bius inversion for uniform matroids, as in Proposition \ref{prop:zeta-uniform}.
\begin{corollary}
\label{cor:upsilon-uniform}
Let $U_{r,n}$ be a uniform matroid. Then
\[
    Y_{U_{r,n}}(s) = \frac{1}{ns + r} \left( -rs \left(\frac{-s}{s+1} \right)^{r-1} {n \choose r}\right).
\]
\end{corollary}

\begin{proof}
    Combining Theorem ~\ref{thm:upsilon-recurrence} and Proposition~\ref{prop:zeta-uniform} gives the result. 
\end{proof}

We also show that the M\"{o}bius inversion is multiplicative over direct sums, as in Proposition \ref{prop:zeta-multiplicative}.

\begin{proposition}
\label{prop:upsilon-multiplicative}
    Let $M_1$ and $M_2$ be matroids. Then
    \[
        Y_{M_1 \oplus M_2}(s) = Y_{M_1}(s) Y_{M_2}(s).
    \]
\end{proposition}

\begin{proof}
    As $\calL(M_1 \oplus M_2) = \calL(M_1) \times \calL(M_2)$, any flat $F \in \calL(M_1 \oplus M_2)$ is the disjoint union $F = F_1 \sqcup F_2$ for $F_1 \in \calL(M_1)$ and $F_2 \in \calL(M_2)$. Thus
    \[
        Y_{M_1 \oplus M_2}(s) = \sum_{\substack{F_1 \in \calL(M_1) \\ F_2 \in \calL(M_2)}} \mu_{M_1 \oplus M_2}(F_1 \sqcup F_2, E_1 \sqcup E_2) Z^{\top}_{M_1 \oplus M_2}(s).
    \]
    Now we see that the M\"{o}bius function is multiplicative in that
    \[
        \mu_{M_1 \oplus M_2}(F_1 \sqcup F_2, E_1 \sqcup E_2) = \mu_{M_1}(F_1, E_1) \mu_{M_2}(F_2, E_2).
    \]
    And as the topological zeta function of the direct sum factors from Proposition \ref{prop:zeta-multiplicative}, we have
    \begin{align*}
        Y_{M_1 \oplus M_2}(s) &= \sum_{\substack{F_1 \in \calL(M_1) \\ F_2 \in \calL(M_2)}} \mu_{M_1}(F_1, E_1) Z^{\top}_{M_1}(s) \cdot \mu_{M_2}(F_2, E_2) Z^{\top}_{M_2}(s) \\
        &= Y_{M_1}(s) Y_{M_2}(s).\qedhere
    \end{align*}
\end{proof}

Finally, we can give a formula for the M\"{o}bius inversion in terms of flags of flats which does not reference the topological zeta function, as in Definition \ref{def:zeta-flags}.

\begin{proposition}
\label{prop:upsilon-flags}
    Let $M$ be a loopless matroid over ground set $E$. The M\"{o}bius inversion of the topological zeta function is given by
    \[
        Y_M(s) = \sum_{\calF} \prod_{i=1}^{\len(\calF)} - \frac{|F_i|s + \rk(F_{i-1})}{|F_i|s + \rk(F_i)},
    \]
    where the sum is taken over all flags $\calF$ in the lattice of flats $\calL(M)$.
\end{proposition}

\begin{proof}
    We proceed by induction on $\rk(M)$. When $\rk(M) = 0$, $M = U_{0,0}$. Here $Y_{U_{0,0}}(s) = 1$ as $Z^{\top}_{U_{0,0}}(s) = 1$: the only flag in $\calL(U_{0,0})$ has length $0$. Now let $r \geqslant 1$, assume towards induction the statement is true for all matroids of rank $< r$, and let $M$ be a matroid over ground set $E$ of rank $\rk(M) = r$. Consider a flag
    \[
        \calF = (\varnothing = F_0 \subsetneq F_1 \subsetneq \dotsc \subsetneq F_{k-1} \subsetneq F_k = E).
    \]
    By hypothesis, there must be some penultimate flat $F_{k-1}$ in every flag. In particular, $F_{k-1} \in \calL(M)$ with $F_{k-1} \neq E$, and
    \[
        \calF_{\leqslant k-1} = (\varnothing = F_0 \subsetneq F_1 \subsetneq \dotsc \subsetneq F_{k-1}),
    \]
    is a flag in $M|F_{k-1}$. Therefore, applying the inductive hypothesis to $Y_{M|F_{k-1}}$, we see that $Y_{M|F_{k-1}}$ will be a sum of terms, one of which will be the product
    \[
        \prod_{i=1}^{\len(\calF_{\leqslant k-1})} - \frac{|F_i|s + \rk(F_{i-1})}{|F_i|s + \rk(F_i)}.
    \]
    Now by applying the recurrence relation in Theorem \ref{thm:upsilon-recurrence}, the term
    \[
        - \frac{|E|s + \rk(F_{k-1})}{|E|s + \rk(E)} Y_{M|F_{k-1}},
    \]
    will contain the product
    \[
        \prod_{i=1}^{\len(\calF)} - \frac{|F_i|s + \rk(F_{i-1})}{|F_i|s + \rk(F_i)}.
    \]
    Thus summing over all proper flats in the recurrence relation will count every flag in $\calL(M)$, so we obtain the desired result.
\end{proof}

\section{Topological zeta functions of matroids under matroid operations}
\label{sec:operations}
In this section, we examine how topological zeta functions behave under two matroid operations: truncation and free extension. In particular, in Theorem \ref{thm:truncation-zeta} and Theorem \ref{thm:free-extension-zeta} we derive formulas for the topological zeta functions of the respective operations in terms of the original topological zeta function and its M\"{o}bius inversion.

\subsection{Truncation} 
We begin by recalling the truncation of a matroid $M$.
\begin{defn}
    \label{defn:truncation}
    The \textbf{truncation} of $M$, denoted $\tr(M)$, is the rank $r-1$ matroid on ground set $E$ defined by the lattice of flats
    \[
        \calL(\tr(M)) = \calL(M) \setminus \calL(M)_{r-1}.
    \]
    Cryptomorphically, the truncation is also defined by the rank function
    \[
        \rk_{\tr(M)}(S) = \min\{\rk_M(S), r-1\}.
    \]
\end{defn}

The next two results are probably well-known to experts. The first relates the characteristic polynomials of $M$ and $\tr(M)$.

\begin{lemma}
    \label{lem:truncation-char-poly}
    Let $M$ be a matroid. The characteristic polynomial of the truncation is
    \[
        \chi_{\tr(M)}(q) = \frac{\chi_M(q) + (q-1)\chi_M(0)}{q}.
    \]
    Moreover
    \[
        \overline{\chi}_{\tr(M)}(q) = \frac{\chi_M(q) + \chi_M(0)}{q}.
    \]
\end{lemma}

\begin{proof}
    Let $M$ be a matroid over ground set $E$ of rank $r$. By definition
    \[
        \chi_{\tr(M)}(q) = \sum_{S \subseteq E} (-1)^{|S|} q^{\rk_{\tr(M)}(\tr(M))-\rk_{\tr(M)}(S)}.
    \]
    We separate the sum in two, based on the rank of the subsets $S \subseteq E$: first the subsets $S$ with $\rk_M(S) \leqslant r-1$, in which case $\rk_{\tr(M)}(S) = \rk_M(S)$, and second the subsets $S$ with $\rk(M) = r$, in which case $\rk_{\tr(M)}(S) = r-1$.
    \begin{align*}
        \chi_{\tr(M)}(q) &= \sum_{\rk_M(S) \leqslant r-1} (-1)^{|S|} q^{\rk_M(M) - \rk_M(S) - 1} + \sum_{\rk_M(S) = r} (-1)^{|S|} \\
        &= \frac{1}{q} \left(\sum_{\rk_M(S) \leqslant r-1} (-1)^{|S|} q^{\rk_M(M) - \rk_M(S)} + q\sum_{\rk_M(S) = r} (-1)^{|S|}\right).
    \end{align*}
    After clearing the $\frac{1}{q}$ term from the $\rk_M(S) \leqslant r-1$ subsets, we can write $q = 1 + (q-1)$ to recover $\chi_M(q)$, with the remaining sum exactly $\chi_M(0)$. Hence
    \begin{align*}
        \chi_{\tr(M)}(q) &= \frac{1}{q} \left(\sum_{S \subseteq E} (-1)^{|S|} q^{\rk_M(M) - \rk_M(S)} + (q-1) \sum_{\rk_M(S) = r} (-1)^{|S|}\right) \\
        &= \frac{\chi_M(q) + (q-1) \chi_M(0)}{q} \qedhere
    \end{align*}
\end{proof}

The following is a standard result relating the truncation of a matroid and the restriction and contraction of flats, which follows immediately from definition. 
\begin{lemma}
\label{lem:truncation-flats}
    Let $M$ be a matroid of rank $r \geqslant 2$, and let $F \in \calL(M)_{\leqslant r-2}$. Then
    \[
        \tr(M)/F = \tr(M/F),
    \]
    and
    \[
        \tr(M)|F = M|F.
    \]
\end{lemma}

    

We can now state our result concerning the topological zeta function of the truncation of a matroid $M$ and its relation to the M\"obius inversion.

\begin{thm}
\label{thm:truncation-zeta}
    Let $M$ be a matroid over ground set $E$. The topological zeta function for the truncation $\tr(M)$ can be written in terms of $Z^{\top}_M(s)$ and its M\"{o}bius inversion $Y_M(s)$ as
    \[
        Z^{\top}_{\tr(M)}(s) = Z^{\top}_M(s) + \frac{1}{|E|s + \rk(M) - 1} Y_M(s).
    \]
\end{thm}

\begin{proof}
    To begin, we use the recurrence relation from Proposition \ref{prop:zeta-recurrence}.
    \[
        Z^{\top}_{\tr(M)}(s) = \frac{1}{|E|s + \rk(M) - 1} \sum_{\substack{F \in \calL(\tr(M)) \\ F \neq E}} \overline{\chi}_{\tr(M)/F}(1) Z^{\top}_{\tr(M)|F}(s)
    \]
    Now using Lemma~\ref{lem:truncation-flats} and Lemma~\ref{lem:truncation-char-poly}, we obtain
    \begin{align*}
        Z^{\top}_{\tr(M)}(s) &= \frac{1}{|E|s + \rk(M) - 1} \sum_{\substack{F \in \calL(\tr(M)) \\ F \neq E}} \overline{\chi}_{\tr(M/F)}(1) Z^{\top}_{M|F}(s) \\
        &= \frac{1}{|E|s + \rk(M) - 1} \sum_{\substack{F \in \calL(\tr(M)) \\ F \neq E}} [\overline{\chi}_{M/F}(1) + \chi_{M/F}(0)] Z^{\top}_{M|F}(s).
    \end{align*}
    Furthermore, if $H \in \calL(M)$ is a hyperplane, then the characteristic polynomial $\chi_{M/H}(q) = q-1$, so $\overline{\chi}_{M/H}(1) + \chi_{M/H}(0) = 0$. Thus we can extend the sum from all proper flats in $\calL(\tr(M))$ to all proper flats in $\calL(M)$.
    \[
        Z^{\top}_{\tr(M)}(s) = \frac{1}{|E|s + \rk(M) - 1} \sum_{\substack{F \in \calL(M) \\ F \neq E}} [\overline{\chi}_{M/F}(1) + \chi_{M/F}(0)] Z^{\top}_{M|F}(s).
    \]
    Then by the reverse application of the recurrence relation in Proposition \ref{prop:zeta-recurrence}
    \[
        Z^{\top}_{\tr(M)}(s) = \frac{|E|s + \rk(M)}{|E|s + \rk(M)-1} Z^{\top}_M(s) + \frac{1}{|E|s + \rk(M) - 1} \sum_{\substack{F \in \calL(M) \\ F \neq E}} \chi_{M/F}(0) Z^{\top}_{M|F}(s).
    \]
    Notice that in the sum, $\chi_{M/F}(0) = \mu_M(F, E)$, and recall that $\mu_M(E,E) = 1$, so we have
    \begin{align*}
        Z^{\top}_{\tr(M)}(s) &= Z^{\top}_M(s) + \frac{1}{|E|s + \rk(M) - 1} \sum_{F \in \calL(M)} \mu_M(E, F) Z^{\top}_{M|F}(s) \\
        &= Z^{\top}_{M}(s) + \frac{1}{|E|s + \rk(M) - 1} Y_M(s).
    \end{align*}
\end{proof}

\subsection{Free Extension}
We now proceed to recall the free extension of a matroid $M$.
\begin{definition}
\label{def:free-extension}
    Let $M$ be a matroid over the ground set $E$ with bases $\calB(M)$ and independent sets $\calI(M)$. The \textbf{free extension} of $M$ by $e$, written $M + e$, is the matroid over the ground set $E \cup e$ defined by the bases
    \[
        \calB(M + e) = \calB(M) \cup \{I \cup \{e\} : I \in \calI(M) \text{ and } |I| = \rk(M) - 1\}.
    \]
    Cryptomorphically, the free extension is defined by the rank function $\rk_{M+e} : 2^{E \cup e} \to \Z_{\geqslant 0}$ such that for any $S \subseteq E$
    \[
        \rk_{M+e}(S) = \rk_M(S),
    \]
    and
    \[
        \rk_{M+e}(S \cup \{e\}) = \begin{cases}
            \rk_M(S) & \text{if } \cl_M(S) = E; \\
            \rk_M(S) + 1 & \text{if } \cl_M(S) \neq E.
        \end{cases}
    \]
\end{definition}

We can recover an expression for the topological zeta function of the free extension of a matroid essentially as a corollary of Theorem~\ref{thm:truncation-zeta}.

\begin{thm}
\label{thm:free-extension-zeta}
    Let $M$ be a matroid over ground set $E$. The topological zeta function for the free extension $M + e$ can be written in terms of $Z^{\top}_M(s)$ and its M\"{o}bius inversion $Y_M(s)$ as
    \[
        Z^{\top}_{M+e}(s) = \frac{1}{s+1} \left(Z^{\top}_M(s) - \frac{s}{(|E|+1)s + \rk(M)} Y_M(s) \right).
    \]
\end{thm}

\begin{proof}
    We note that $M + e = \tr(M \oplus U_{1,1})$. Therefore, because the topological zeta function is multiplicative over direct sums, we apply Theorem \ref{thm:truncation-zeta} to $Z^{\top}_M(s) \cdot \frac{1}{s+1}$.
\end{proof}

\section{Taylor coefficients of topological zeta functions of matroids}
\label{sec:taylor}
In this section, we prove a result that characterizes the low order Taylor coefficients of the topological zeta function in terms of what we call the ``girth'' of a matroid. We explain how the girth can be related to uniform matroids via truncation, and propose a conjecture about the Taylor coefficients of truncated matroids. We begin with the definition.

\begin{definition}
\label{def:girth}
    Let $M$ be a matroid over ground set $E$. We define the \textbf{girth} $\gir(M)$ of a matroid to be the size of the smallest circuit in $M$; if a matroid has no circuits, then we say $\gir(M) = |E| + 1$.
\end{definition}

Note that if $M_G$ is a graphic matroid, then $\gir(M_G)$ corresponds to the girth of the underlying graph $G$, i.e., the size of the smallest cycle present in $G$. 

We now state the main result in this section.
\begin{thm}
\label{thm:girth}
    Let $M$ be a matroid over ground set $E$ of girth $g > 0$. Then for all integers $0 \leqslant k < g$, the derivative of the topological zeta function evaluated at $0$ is
    \[
        \left(\frac{d^k}{ds^k}Z^{\top}_M(s) \right) \bigg|_{s=0} = (-1)^k |E|^{\overline{k}},
    \]
    where $|E|^{\overline{k}} = |E|(|E|+1) \cdots(|E| + k-1)$ is the rising $k^{th}$ factorial of $|E|$.
\end{thm}

\begin{remark}
    Note that the Taylor coefficient in Theorem~\ref{thm:girth} is precisely the Taylor coefficient of the topological zeta function of a uniform matroid (Proposition~\ref{prop:zeta-uniform}). That is, if $M$ is a matroid of rank $r$ and the truncation $\tr(M)$ is a uniform matroid, then for all integers $0 \leqslant k < r$ we have
    \[
        \left(\frac{d^k}{ds^k} Z^{\top}_M(s)\right) \bigg|_{s=0} = \left(\frac{d^k}{ds^k} Z^{\top}_{\tr(M)}(s) \right) \bigg|_{s=0}.
    \]
\end{remark}

This motivates the following conjecture, which can be seen as a generalization of Theorem \ref{thm:girth} to all truncations, not just uniform truncations.

\begin{conjecture}
    \label{conj:truncation}
    Let $M$ be a matroid of rank $r > 0$. Then for all integers $0 \leqslant k < \rk(M)$ we have
    \[
        \left(\frac{d^k}{ds^k} Z^{\top}_M(s)\right) \bigg|_{s=0} = \left(\frac{d^k}{ds^k} Z^{\top}_{\tr(M)}(s) \right) \bigg|_{s=0}.
    \]
\end{conjecture}

\begin{remark}
    Note that this would imply that if two matroids $M_1$ and $M_2$ have the same lattice up to a certain rank, i.e. $\calL(M_1)_{\leqslant r} = \calL(M_2)_{\leqslant r}$, then the first $r$ Taylor coefficients of $Z^{\top}_{M_1}(s)$ and $Z^{\top}_{M_2}(s)$ would be the same. In this sense, the low order Taylor coefficients of $Z^{\top}_M(s)$ would encode information about the lattice $\calL(M)$.
\end{remark}

In addition, as Theorem \ref{thm:truncation-zeta} describes the topological zeta function of the truncation in terms of the M\"{o}bius inversion $Y_M(s)$, we propose a stronger conjecture about the behavior of the $Y_M$ function which implies Conjecture \ref{conj:truncation}.

\begin{conjecture}
    \label{conj:upsilon}
    Let $M$ be a matroid of rank $r$. Then the Taylor expansion of $Y_M(s)$ centered at $s=0$ is given by
    \[
        Y_M(s) = (-1)^r |\calB(M)| s^r + O(s^{r+1}).
    \]
\end{conjecture}

Now we will develop the technical tools we will need to prove Theorem \ref{thm:girth}. First, we recall Stirling numbers of the first and second kind, and state some basic facts about them.
\begin{definition}
    The \textbf{(unsigned) Stirling numbers of the first kind, $c(n,k)$}, count the number of permutations in the symmetric group $S_n$ with $k$ disjoint cycles. When $n = k = 0$, we set $c(0,0) = 1$. The \textbf{Stirling numbers of the second kind, $S(n,k)$}, count the number of partitions of the set $[n]$ into $k$ parts.
\end{definition}

\begin{fact}
\label{fact:stirling-rising-factorial}
    Let $n, k \geqslant 0$ be integers.
    The rising factorial $n^{\overline{k}}$ can be written as
    \[
       n^{\overline{k}} = \sum_{i=1}^k c(k, i)x^i.
    \]
\end{fact}

\begin{fact}
\label{fact:stirling-recurrence}
    Let $n, k \geqslant 0$ be integers. $c(n,k)$ satisfies the recurrence relation
        \[
            c(n,k) = (n-1) c(n-1,k) + c(n-1,k-1).
        \]
\end{fact}

\begin{fact}[{\cite[Section 6.1]{knuth1994concrete}}]
\label{fact:stirling-both}
Let $n, k \geqslant 1$ be integers. The Stirling numbers of both kinds and the falling factorials are related by the equality
\[
    \sum_{n=m}^n c(n,k) S(k,m) = {n \choose m} (n-1)^{\underline{n-m}}.
\]
\end{fact}

The next lemma will play a key roll in our proof of Theorem \ref{thm:girth}.

\begin{lemma}
\label{lemma:stirling-both}
    Let $1 \leqslant j \leqslant k$ be positive integers. Then
    \begin{equation*}
    \label{equ:stirling}
        j \sum_{i=j}^k c(k,i) S(i,j) = k \sum_{i=j}^k c(k-1,i-1)S(i,j)
    \end{equation*}
\end{lemma}

\begin{proof}
    When $i < j$, the Stirling number of the second kind $S(i,j) = 0$ by definition. Thus we can rewrite both sums with $i$ ranging from $1$ to $k$. We first consider the sum on the right hand side
    \[
        R = k \sum_{i=1}^k c(k-1, i-1) S(i,j).
    \]
    Using the recurrence relation in Fact \ref{fact:stirling-recurrence}, we can rewrite this as
    \[
    R = k \sum_{i=1}^k c(k,i) S(i,j) - (k-1) \sum_{i=1}^k c(k-1,i) S(i,j).
    \]
    Now from the identity in Fact \ref{fact:stirling-both}, these sums are precisely
    \[
        R = k {k \choose j} (k-1)^{\underline{k-j}} - k(k-1) {k-1 \choose j} (k-2)^{\underline{k-j-1}}.
    \]
    Grouping the falling factorials together, this becomes
    \[
        R = k \left[{k \choose j} - {k-1 \choose j} \right] (k-1)^{\underline{k-j}} = k {k-1 \choose j-1} (k-1)^{\underline{k-j}}.
    \]
    On the other hand, applying Fact \ref{fact:stirling-both} to the left hand side directly gives us
    \[
        L = j {k \choose j} (k-1)^{\underline{k-j}}.
    \]
    So we conclude that $L = R$ because $j {k \choose j} = k {k-1 \choose j-1}$.
\end{proof}

Now we state some results related the ground set $E$ of a matroid $M$ to the number of sets of a given size and rank.

\begin{definition}
\label{def:matroid-sets}
    Let $M$ be a matroid over ground set $E$, and suppose that $r$ and $s$ are positive integers. We denote $\calD^r_s(M)$ to be the collection of subsets of $E$ of rank $r$ and size $s$, that is
    \[
        \calD^r_s(M) = \{S \subseteq E : |S| = s \text{ and } \rk_M(S) = r\}.
    \]
    Note that for some values of $r$ and $s$, we may have $\calD^r_s(M) = \varnothing$.
\end{definition}

\begin{lemma}
\label{lemma:sum-Ds}
    Let $M$ be a loopless matroid over ground set $E$. Then for any positive integer $1 \leqslant s \leqslant |E|$ we have
    \[
        {|E| \choose s} = \sum_{r=1}^s |\calD^r_s(M)|.
    \]
\end{lemma}

\begin{proof}
    Every subset $S \subset E$ where $|S| = s$ must have rank $1 \leqslant \rk(S) \leqslant s$ as $M$ is loopless. Thus the collection of sets $\{\calD^r_s(M)\}$ for $1 \leqslant r \leqslant s$ partitions the ${|E| \choose s}$ cardinality $s$ subsets of $E$.
\end{proof}

\begin{lemma}
\label{lemma:power-E}
    Let $M$ be a loopless matroid over ground set $E$ and $k$ be a positive integer. Then
    \[
        |E|^k = \sum_{j=1}^k j! S(k,j) \sum_{i=1}^j |\calD^i_j(M)|.
    \]
\end{lemma}

\begin{proof}
    Consider the mapping $\varphi : E^k \to 2^E$ defined as $(e_1, \dotsc, e_k) \mapsto \{e_1, \dotsc, e_k\}$ which sends a $k$-tuple to the subset of corresponding elements in $E$. The image of $\varphi$ is the powerset $2^E \setminus \varnothing$. We observe that, for any $S \subset E$ of size $|S| = j$, the preimage $\varphi^{-1}(S)$ corresponds exactly to the set of all possible surjections $[k] \to [j]$. Each such surjection induces a partition of $[k]$ into $j$ equivalence classes, of which there are $S(k,j)$ partitions and $j!$ ways to assign the equivalence classes to distinct elements of $[j]$. Therefore we can write $E^k$ as the disjoint union of all such preimages, ranging over all $1 \leqslant j \leqslant k$ and all subsets $S \subset E$ of size $|S| = j$. This gives us the sum
    \[
        |E|^k = \sum_{j=1}^k (j!)\: S(k,j) {|E| \choose j}.
    \]
    The result then follows from applying Lemma \ref{lemma:sum-Ds} to expand each binomial coefficient ${|E| \choose j}$ as desired.
\end{proof}

Now we recall some relevant definitions and state some results from \cite{jensen2021motivic} which are relevant to our work.

\begin{definition}
    Let $n \geqslant 0$ be an integer. We define the \textbf{$q$-analogue} to be the polynomial
    \[
        [n]_q = \frac{q^n - 1}{q-1} = 1 + q + \dotsc + q^{n-1} \in \Z[q].
    \]
    Note that if $n = 0$ then $[n]_q = 0$, and if $q = 1$ then $[n]_q = n$.
\end{definition}

\begin{definition}
    Let $M$ be a loopless matroid over ground set $E$. We define the \textbf{reduced lattice} $\hat{\calL}(M) = \calL(M) \setminus \{\varnothing, E\}$. We will write $\hat{\calL}$ when the context is clear.
\end{definition}

\begin{proposition}[\cite{jensen2021motivic}, Proposition 2.2]
\label{prop:two-flats}
    For any two flats $F_1 \subseteq F_2$ in $M$,
    \[
        [\rk(M) - \rk(F)]_q = \sum_{F_1 \subseteq F \subsetneq F_2} \overline{\chi}_{M|F_2/F}(q).
    \]
\end{proposition}

\begin{corollary}
\label{cor:flat-diff}
    Let $M$ be a matroid over ground set $E$ with a flat $F$. Then
    \[
        [\rk(M) - \rk(F)]_q = \sum_{F \subseteq F' \in \hat{\calL}} \overline{\chi}_{M/F'}(q).
    \]
\end{corollary}

\begin{proof}
    Let $F_1 = F$ and $F_2 = E$ in Proposition \ref{prop:two-flats}.
\end{proof}

Now we prove two lemmas which compute sums we will encounter in our proof of Theorem \ref{thm:girth}.

\begin{lemma}
\label{lemma:flat-sum}
    Let $M$ be a loopless matroid over ground set $E$ and suppose that $i, j \geqslant 1$ are integers. Then
    \[
        \sum_{F \in \hat{\calL}} \overline{\chi}_{M/F}(q) |\calD^i_j(M |F) | = |\calD^i_j(M)| \cdot [\rk(M) - i]_q.
    \]
\end{lemma}

\begin{proof}
    When $i > \rk(M)$, the statement is immediate since both sides are $0$, as each $|\calD^i_j| = 0$. When $i = \rk(M)$, the left hand side is $0$ vacuously because the sum is empty, and the right hand side is $0$ because $[\rk(M)-\rk(M)]_q = [0]_q = 0$. Therefore, we can assume that $i < \rk(M)$. Expanding the sum, we obtain
    \begin{align*}
        \sum_{F \in \hat{\calL}} \overline{\chi}_{M/F}(q) |\calD^i_j(M |F)| &= \sum_{F \in \hat{\calL}} \sum_{A \in \calD^i_j(M|F)} \overline{\chi}_{M/F}(q) \\
        &= \sum_{A \in \calD^i_j(M)} \sum_{\substack{F \in \hat{\calL} \\ A \subseteq F}} \overline{\chi}_{M/F}(q) \\
        &= \sum_{A \in \calD^i_j(M)} [\rk(M) - i]_q \\
        &= |\calD^i_j(M)| \cdot [\rk(M)-i]_q.
    \end{align*}
    Here the second to last equality follows from Corollary \ref{cor:flat-diff}, and the rest follow from reindexing the sum.
\end{proof}

\begin{lemma}
\label{prop:flat-sum}
    Let $M$ be a loopless matroid and $k \geqslant 1$ be an integer. Then
    \[
        \sum_{F \in \hat{\calL}} \overline{\chi}_{M/F}(q) |F|^k = \sum_{j=1}^k j! S(k,j) \sum_{i=1}^j |\calD^i_j(M)| \cdot [\rk(M)-i]_q.
    \]
\end{lemma}

\begin{proof}
    Applying the result of Lemma \ref{lemma:power-E} to each of the flats $F \in \hat{\calL}$, we see that
    \[
        \sum_{F \in \hat{\calL}} \overline{\chi}_{M/F}(q) |F|^k = \sum_{F \in \hat{\calL}} \overline{\chi}_{M/F}(q) \sum_{j=1}^k j! S(k,j) \sum_{i=1}^j |\calD^i_j(M | F)|.
    \]
    Exchanging the order of summation of $F$ with the sums in $j$ and $i$, we can apply the result of Lemma \ref{lemma:flat-sum},
    \begin{align*}
        \sum_{F \in \hat{\calL}} \overline{\chi}_{M/F}(q) |F|^k &= \sum_{j=1}^k j! S(k,j) \sum_{i=1}^j \sum_{F \in \hat{\calL}} \overline{\chi}_{M/F}(q) |\calD^i_j(M | F)| \\
        &= \sum_{j=1}^k j! S(k,j) \sum_{i=1}^j |\calD^i_j(M)| [\rk(M) - i]_q. \qedhere
    \end{align*}
\end{proof}

Our final lemma expresses the $k$th derivative of the topological zeta function, and will be the basis of our proof of Theorem \ref{thm:girth}.

\begin{lemma}
\label{prop:k-derivative}
    Let $M$ be a loopless matroid with ground set $E$. Then
    \[
        \frac{d^k}{ds^k} Z^{\top}_M(s) = \frac{1}{|E|s+\rk(M)} \left(-k|E| \frac{d^{k-1}}{ds^{k-1}} Z^{\top}_M(s) + \sum_{F \in \hat{\calL}} \overline{\chi}_{M/F}(1) \frac{d^k}{ds^k} Z^{\top}_{M|F}(s) \right).
    \]
\end{lemma}

\begin{proof}
    The result follows from induction and the quotient rule applied to the recurrence formula for the topological zeta function in Proposition~\ref{prop:zeta-recurrence}. 
\end{proof}

Now we are finally ready to prove Theorem \ref{thm:girth}.

\begin{proof}[Proof of Theorem \ref{thm:girth}]
    We fix $g$ and proceed by induction on $|E|$. The base case is a uniform matroid, which is handled by Proposition~\ref{prop:zeta-uniform}. For the inductive step, assume that we have a matroid $M$ of girth $\gir(M)\geqslant g$ over a ground set $E$, and that all matroids with smaller cardinality ground sets have the desired first $g-1$ Taylor coefficients.


    Within this inductive step, we also use induction on $k$, the order of the derivative. As $Z^{\top}_M(0) = 1$ for all matroids, the base case $k = 0$ holds. Now for any $0 < k < g$, assume towards induction that the first $(k-1)$ Taylor coefficients are as desired. By Lemma \ref{prop:k-derivative}, the $k$-th derivative evaluated at $s = 0$ is
    \begin{equation*}
        \label{ds:eq:kth-der}
            \left(\frac{d^k}{ds^k} Z^{\top}_M(s) \right) \Bigg|_{s=0} = \frac{1}{\rk(M)} \Bigg[\overbrace{-k|E| \left(\frac{d^{k-1}}{ds^{k-1}} Z^{\top}_M(s)\right) \Bigg|_{s=0}}^{A}+ \overbrace{\sum_{F \in \hat{\calL}} \overline{\chi}_{M / F}(1) \left(\frac{d^k}{ds^k} Z^{\top}_{M |F}(s) \right) \Bigg|_{s=0}}^{B} \Bigg].
    \end{equation*}
    We will analyze the two major terms in this expression, which we label $A$ and $B$. First, by our inductive hypothesis (on $k$) and Fact \ref{fact:stirling-rising-factorial}, we see
    \[
        A = (-1)^k k |E| \cdot |E|^{\overline{k-1}} = (-1)^k k \sum_{i=1}^{k-1} c(k-1, i) |E|^{i+1}.
    \]
    Note that because $c(k-1, 0) = 0$, we can reindex our sum to the more natural expression
    \[
        A = (-1)^k k \sum_{i=1}^k c(k-1, i-1) |E|^i.
    \]
    Now we can apply Lemma \ref{lemma:power-E} to rewrite $|E|^i$ for each $i$. In particular, because of the assumption on $\gir(M)$, we know that $|\calD^i_j(M)| = 0$ unless $i = j$. This gives the expression
    \[
        A = (-1)^k k \sum_{i=1}^k c(k-1,i-1) \sum_{j=1}^i j! S(i,j) |\calD^j_j(M)|.
    \]
    Our last step here will be to exchange the order of summation of $i$ and $j$. This will allow us to express $A$ in a form convenient for the use of Lemma \ref{lemma:stirling-both} later on:
    \[
        A = (-1)^k \sum_{j=1}^k j! |\calD^j_j(M)| \cdot k \sum_{i=j}^k c(k-1,i-1) S(i,j).
    \]
    Putting this aside for now, we turn to the second term, $B$. For each $F \in \hat{\calL}$, we can apply the induction hypothesis (on $|E|$) to $M|F$. Then applying Fact \ref{fact:stirling-rising-factorial}
    \[
        B = (-1)^k \sum_{F \in \hat{\calL}} \overline{\chi}_{M/F}(1) |F|^{\overline{k}} = (-1)^k \sum_{F \in \hat{\calL}} \overline{\chi}_{M/F}(1) \sum_{i=1}^k c(k,i) |F|^i.
    \]
    After exchanging the order of summation of $i$ and $F$, we can apply Lemma \ref{prop:flat-sum} to obtain
    \[
        B = (-1)^k \sum_{i=1}^k c(k,i) \sum_{j=1}^i j! S(i,j) |\calD^j_j(M)| \cdot (\rk(M) - j).
    \]
    We will group this into two terms, one containing $\rk(M)$ and one containing $j$ from the rightmost term in the sum. That is, we write $B = C - D$, where
    \[
        C = (-1)^k \rk(M) \sum_{i=1}^k c(k,i) \sum_{j=1}^i j! S(i,j) |\calD^j_j(M)|,
    \]
    and
    \[
        D = (-1)^k \sum_{i=1}^k c(k,i) \sum_{j=1}^i j! \cdot j S(i,j) |\calD^j_j(M)|.
    \]
    By the reverse application of Lemma \ref{lemma:power-E} and Fact \ref{fact:stirling-rising-factorial}, we see that
    \[
        C = (-1)^k \rk(M) |E|^{\overline{k}}.
    \]
    We claim that $D = A$, which is sufficient to prove the result since then
    \[
        \left( \frac{d^k}{ds^k} Z^{\top}_M(s) \right) \bigg|_{s=0} = \frac{A + C - D}{\rk(M)} = \frac{C}{\rk(M)} = (-1)^k |E|^{\overline{k}}.
    \]
    By switching the order of summation of $i$ and $j$, we can rewrite $D$, again with the intention of applying Lemma \ref{lemma:stirling-both}.
    \[
        D = (-1)^k \sum_{j=1}^k j! |\calD^j_j(M)| \cdot j \sum_{i=j}^k c(k,i) S(i,j).
    \]
    Now by comparing the expressions for $A$ and $D$, we see that it is sufficient to show that for any $1 \leqslant j \leqslant k$
    \[
        j \sum_{i=j}^k c(k,i) S(i,j) = k \sum_{i=j}^k c(k-1,i-1) S(i,j).
    \]
    But this is precisely the statement of Lemma \ref{lemma:stirling-both}.
\end{proof}

\bibliographystyle{alpha}
\bibliography{bibliography.bib}

@article{jensen2021motivic,
  title={The motivic zeta functions of a matroid},
  author={Jensen, David and Kutler, Max and Usatine, Jeremy},
  journal={Journal of the London Mathematical Society},
  volume={103},
  number={2},
  pages={604--632},
  year={2021},
  publisher={Wiley Online Library}
}

@article{denef1992caracteristiques,
  title={Caract{\'e}ristiques d'Euler-Poincar{\'e}, fonctions z{\^e}ta locales et modifications analytiques},
  author={Denef, Jan and Loeser, Fran{\c{c}}ois},
  journal={Journal of the American Mathematical Society},
  pages={705--720},
  year={1992},
  publisher={JSTOR}
}

@article{feichtner2004incidence,
  title={Incidence combinatorics of resolutions},
  author={Feichtner, Eva-Maria and Kozlov, Dmitry N},
  journal={Selecta Mathematica},
  volume={10},
  number={1},
  pages={37},
  year={2004},
  publisher={Springer}
}

@article{van2019combinatorial,
  title={Combinatorial analogs of topological zeta functions},
  author={van der Veer, Robin},
  journal={Discrete Mathematics},
  volume={342},
  number={9},
  pages={2680-2693},
  year={2019},
  publisher={Elsevier}
}

@book{white1987combinatorial,
author = {Neil White}, place={Cambridge}, series={Encyclopedia of Mathematics and its Applications}, title={Combinatorial Geometries}, DOI={10.1017/CBO9781107325715}, publisher={Cambridge University Press}, year={1987}, collection={Encyclopedia of Mathematics and its Applications}}

@book{knuth1994concrete,
author = {Graham, Ronald L. and Knuth, Donald E. and Patashnik, Oren},
title = {Concrete Mathematics: A Foundation for Computer Science},
year = {1994},
isbn = {0201558025},
publisher = {Addison-Wesley Longman Publishing Co., Inc.},
address = {USA},
edition = {2nd}
}

@book {oxley2011book,
    AUTHOR = {Oxley, James},
     TITLE = {Matroid theory},
    SERIES = {Oxford Graduate Texts in Mathematics},
    VOLUME = {21},
   EDITION = {Second},
 PUBLISHER = {Oxford University Press, Oxford},
      YEAR = {2011},
     PAGES = {xiv+684},
      ISBN = {978-0-19-960339-8},
   MRCLASS = {05-01 (05B35 90C27)},
  MRNUMBER = {2849819},
MRREVIEWER = {Maruti M. Shikare},
       DOI = {10.1093/acprof:oso/9780198566946.001.0001},
       URL = {https://doi-org.ezproxy.uky.edu/10.1093/acprof:oso/9780198566946.001.0001},
}

@article{vanderveer2019,
   title={Combinatorial analogs of topological zeta functions},
   volume={342},
   ISSN={0012-365X},
   url={http://dx.doi.org/10.1016/j.disc.2019.05.035},
   DOI={10.1016/j.disc.2019.05.035},
   number={9},
   journal={Discrete Mathematics},
   publisher={Elsevier BV},
   author={van der Veer, Robin},
   year={2019},
   month={Sep},
   pages={2680?2693}
}

\end{document}